\documentclass{amsart}

\usepackage{amsmath,amsxtra,amsthm,amssymb,amsfonts,graphicx} 

\usepackage[pdfusetitle,bookmarks=true,bookmarksnumbered=false,bookmarksopen=false,breaklinks=false,pdfborder={0 0 0},backref=false,colorlinks=true]{hyperref}

\theoremstyle{plain}

\newtheorem{thm}{Theorem}[section]

\newtheorem{prop}[thm]{Proposition}
\newtheorem{lem}[thm]{Lemma}
\newtheorem{cor}[thm]{Corollary}

\theoremstyle{definition}

\newtheorem{rem}[thm]{Remark}

\def\l{\left}
\def\r{\right}
\def\lang{\l\langle}
\def\rang{\r\rangle}

\def\as{\!\mathrel{\mathop:}=}

\def\fun{\col X\to\exrls}
\def\wgt{\col [0,\infty)\to[0,\infty)}
\def\mdl{\col (0,\infty)\to(0,\infty)}
\def\map{\col X\to X}

\def\nat{\mathbb{N}}

\def\rls{\mathbb{R}}
\def\exrls{(-\infty,\infty]}
\def\eps{\varepsilon}
\def\lam{\lambda}

\def\half{\frac{1}{2}}
\newcommand{\hlf}[1]{\frac{#1}{2}}
\newcommand{\rec}[1]{\frac{1}{#1}} 
\newcommand{\nnorm}[1]{{\left\vert\kern-0.25ex\left\vert\kern-0.25ex\left\vert #1 
    \right\vert\kern-0.25ex\right\vert\kern-0.25ex\right\vert}}

\def\prox{\mathrm{prox}}
\def\aprox{\widehat{\mathrm{prox}}}

\def\col{\colon}
\def\ol{\overline}

\def\argmin{\operatornamewithlimits{\arg\min}}

\def\dom{\operatorname{dom}}
\def\ch{\operatorname{Ch}}
\def\cldom{\operatorname{\ol{dom}}}

\def\di{\operatorname{d}\!}

\begin{document}
\title[Proximal mappings with Young functions]{On proximal mappings with Young functions in uniformly convex Banach spaces}

\date{\today}
\thanks{}

\author[M. Ba\v{c}\'ak]{Miroslav Ba\v{c}\'ak}
\address{Max Planck Institute for Mathematics in the Sciences, Inselstr. 22, 04 103 Leipzig, Germany}
\email{bacak@mis.mpg.de}

\author[U. Kohlenbach]{Ulrich Kohlenbach}
\address{Department of Mathematics, Technische Universit\"{a}t Darmstadt, Schlossgartenstra{\ss}e 7, 64289 Darmstadt, Germany}
\email{kohlenbach@mathematik.tu-darmstadt.de}

\begin{abstract}
It is well known in convex analysis that proximal mappings on Hilbert spaces are $1$-Lipschitz. In the present paper we show that proximal mappings on uniformly convex Banach spaces are uniformly continuous on bounded sets. Moreover, we introduce a new general proximal mapping whose regularization term is given as a composition of a Young function and the norm, and formulate our results at this level of generality. It is our aim to obtain the corresponding modulus of uniform continuity explicitly in terms of a modulus of uniform convexity of the norm and of moduli witnessing properties of the Young function. We also derive several quantitative results on uniform convexity, which may be of interest on their own. 
\end{abstract}

\maketitle
\section{Introduction}

Let $X$ be a Hilbert space and $f\fun$ be a (proper) convex lower semicontinuous function. The proximal mapping associated with~$f$ is given by
\begin{equation*}
 \prox_f(x)\as\argmin_{y\in X} \l[f(y)+\half\|x-y\|^2  \r], \quad x\in X.
\end{equation*}
This definition is due to Moreau \cite{moreau62,moreau63,moreau65}. Given $\lam>0,$ it is common to consider
\begin{equation} \label{eq:pm2}
 \prox_{\lam f}(x)=\argmin_{y\in X} \l[f(y)+\rec{2\lam}\|x-y\|^2  \r],
\end{equation}
which is exactly the resolvent of the maximal monotone operator $\partial f,$ that is, 
\begin{equation*}
 \prox_{\lam f}(x)=\l(I+\lam\partial f\r)^{-1}(x),
\end{equation*}
for every $x\in X.$ The proximal mapping is nonexpansive, that is, Lipschitz continuous with Lipschitz constant~$1.$ A prime example of the proximal mapping is the metric projection onto a closed convex set $C\subset~X$ and it is known that the nonexpansiveness of metric projections onto closed convex sets characterizes Hilbert spaces among Banach spaces provided the dimension is at least three; this result is due to Phelps \cite[Theorem 5.2]{phelps57}. More details on the proximal mapping in Hilbert spaces can be found for instance in the monograph by Bauschke and Combettes~\cite{bc}. We also recommend the classic books \cite{brezis} and \cite{phelps} by Brezis and Phelps, respectively.

Let now $X$ be a uniformly convex Banach space. Given a closed convex set $C\subset X,$ we use the standard notation $P_C$ and $d_C$ for the metric projection and distance function, respectively. Let also $B(x,r)$ stand for the closed ball with diameter $r>0$ centered at $x\in X.$ In particular, we denote $B_X\as B(0,1).$ It is well known that~$P_C$ is continuous; see for instance the book by Goebel and Reich \cite[Proposition 3.2]{gr}. The following quantitative result is from the monograph by Benyamini and Lindenstrauss \cite[Lemma 2.5]{bl}. Assume $C\subset X$ is convex closed, $x\in X$ with $d_C(x)<R,$ for some $R>0,$ and $y\in B(x,r),$ for some $r\in(0,R).$ Then we have
\begin{equation*}
 \l\|P_C(x)-P_C(y) \r\|\le \l(R+r\r) \delta_X^{-1}\l(\frac{2r}{R+r} \r).
\end{equation*}
Here $\delta_X$ stands for a modulus of uniform convexity of~$X;$ see Section~\ref{sec:pre} for the definition.

Penot \cite[Theorem 4.5]{penot} obtained a similar result which we will now present. Let us establish the relevant terminology first.

A continuous strictly convex function $\Phi\wgt$ is a \emph{Young function} if it satisfies
\begin{equation} \label{eq:young}
 \lim_{t\to0}\frac{\Phi(t)}{t} =0, \quad\text{and}\quad \lim_{t\to\infty}\frac{\Phi(t)}{t}  = \infty.
\end{equation}
Then there exists a right-continuous strictly increasing function $\phi\wgt$ with $\phi(0)=0$ and $\lim_{t\to\infty}\phi(t)=\infty$ such that
\begin{equation*}
 \Phi(t)\as\int_0^t \phi(s) \di s.
\end{equation*}
The class of Young functions include $\rec{p}t^p,$ for $p\ge2,$ $e^t-t-1,$ $\cosh(t)-1,$ $t^p\log\l(t+1\r)^r,$ for $p\ge1$ and $r>0.$ For more details, see for instance the classic book by Krasnoselski and Ruticki~\cite{kr}.

Consider a duality mapping $J_\phi\col X\to 2^{X^*},$ given by
\begin{equation} \label{eq:dualitymapping}
 J_\phi(x)\as \l\{x^*\in X^*\col \phi(\|x\|)=\l\|x^*\r\|, \lang x^*,x\rang=\|x\|\l\|x^*\r\| \r\}.
\end{equation}
Recall that the concept of a duality mapping goes back to Beurling and Livingston~\cite{beurlinglivingston}.

We can now state the promised Penot's theorem.
\begin{thm}[Penot]
 Let $X$ be a uniformly convex Banach space and $C\subset X$ be a closed convex set. Let $R>d(0,C)$ and assume that there exists a~nondecreasing function $\gamma\wgt$ satisfying $\gamma(t)>0,$ for each $t>0,$ such that
 \begin{equation*}
  \lang x^*-y^*,x-y\rang \ge \gamma\l(\|x-y\| \r),
 \end{equation*}
for every $x,y\in B(0,R)$ and $x^*\in J_\phi(x)$ and $y^*\in J_\phi(y).$ Then for every $r>0$ with $3r+d(0,C)<R$ we have that the metric projection $P_C$ is uniformly continuous on $B(0,r)$ with 
\begin{equation*}
 \l\| P_C(x)-P_C(y) \r\| \le \gamma^{-1}\l(2\phi(R)\|x-y\|  \r) + \|x-y\|,
\end{equation*}
for every $x,y$ in $B(0,r).$
\end{thm}

In the present paper, we investigate proximal mappings in uniformly convex Banach spaces associated to a (proper) convex lower semicontinuous (lsc, for short) function $f\fun.$ As a matter of fact we introduce a new type of proximal mappings, whose regularization term is given as a composition of a~Young function and the norm: 
\begin{equation} \label{eq:pmy}
 \prox_{\lam,f}^{\Phi}(x)\as\argmin_{y\in X} \l[f(y)+\rec{\phi(\lam)}\Phi\l(\|x-y\| \r) \r],\quad x\in X,
\end{equation}
where $\lam>0.$ For $\Phi(t)\as \rec{p}t^p,$ where $p\ge2,$ we recover
\begin{equation} \label{eq:ppm} 
 \prox_{\lam,f}^{\Phi}(x)=\argmin_{y\in X} \l[f(y)+\rec{p\lam^{p-1}}\|x-y\|^p \r],
\end{equation}
which was used, for instance, by Ambrosio, Gigli and Savar\'e in~\cite{ags} to construct gradient flows ($p$-curves) in metric spaces; see \cite[Remark 2.0.7]{ags}. The proximal mapping~\eqref{eq:ppm} was also used in $p$-uniformly convex metric spaces by Kuwae~\cite{kuwae} as a tool in the study of $p$-harmonic mappings. A corresponding definition of $p$-Moreau envelopes (infimal convolutions) in Hilbert spaces can be found in \cite[Proposition 12.15]{bc}.

Let us turn back to duality mappings now. In the simplest case, when $\phi(t)\as t,$ one obtains the \emph{normalized} duality mapping $J\col X\to 2^{X^*}$ given by
\begin{equation*}
 J(x)\as \l\{x^*\in X^*\col \|x\|=\l\|x^*\r\|, \lang x^*,x\rang=\|x\|\l\|x^*\r\| \r\}.
\end{equation*}
However, in some cases, the duality mapping associated to the function $\phi(t)\as t^{p-1}$ may reflect better the geometry of the underlying Banach space than the normalized duality mapping. Indeed, Penot and Ratsimahalo~\cite{pr} show that if $X\as L_p(\Omega),$ for some $p\in(1,\infty),$ then the duality mapping $J_\phi$ with $\phi(t)\as t^{p-1}$ has a simple form
\begin{equation*}
 J_\phi(x)(\omega) = \l|x(\omega) \r|^{p-2} x(\omega),\qquad \text{a.e. }\omega\in\Omega,
\end{equation*}
which, unlike the normalized duality mapping, does not involve integration. The general duality mapping~\eqref{eq:dualitymapping} is then believed to be more suitable for other Banach spaces (e.g. Orlicz spaces) and so is the proximal mapping with a Young function~\eqref{eq:pmy}. One can also extend definition~\eqref{eq:pmy} to nonlinear spaces, which may be natural in metric space generalizations of Orlicz spaces \cite{kell14,sturm} and in minimization of functionals like the generalized Cheeger energy
\begin{equation*}
 \ch_\Phi(g)\as\int \Phi\l(\l|\nabla g \r|\r)\di\mu,
\end{equation*}
introduced along with the corresponding generalized Laplacian $\Delta_\Phi$ by Kell~\cite{kell16}.

To our knowledge, the continuity of proximal mappings in uniformly convex Banach spaces has not been addressed in the literature 
(apart from the case of metric projections mentioned above) and our results are new even for the case $\phi(t)\as t.$
We also provide explicit moduli of uniform continuity of the proximal mapping depending on a modulus of uniform 
convexity of the underlying space~$X$ and on moduli witnessing properties of the function~$\Phi.$ On the other hand, remarkably, this modulus of uniform continuity of $\prox_{\lam,f}^{\Phi}$ is \emph{independent} of $\lam$ for $\lam\in(0,1].$ Furthermore, we establish several quantitative results on uniform convexity which may be of interest on their own. In some cases, their non-quantitative versions had existed and we obtained our results by extracting additional 
information from the original proofs. This approach is part of a general program of obtaining statements 
with explicit effective bounds by applying proof-theoretic methods developed by the second author. 
However, in the present paper, we do not discuss the underlying principles from logic and proof theory and 
instead refer the interested reader to \cite{kohlenbach16,kohlenbach05,kohlenbach-book} for more information. 
One of the consequences of our methodology is that we work with \emph{nonoptimal} moduli. For instance, it is 
common in functional analysis to define {\bf the} (optimal) modulus of uniform convexity of a Banach space by 
$\delta_X(\eps)\as \inf\l\{1-\l\|\hlf{x+y}\r\|\col x,y\in B_X,\|x-y\|\ge\eps \r\},$ whereas we prefer 
to call an \emph{arbitrary} function~$\delta_X$ witnessing  
\begin{equation*}
\forall \eps\in (0,2]\ \exists\delta >0 \ A(\eps,\delta) 
\end{equation*}
(i.e. any so-called `Skolem function' for this property), 
where 
\begin{equation*}
 A(\eps,\delta)\as\forall x,y\in B_X \ \l(\l\| (x+y)/2\r\| >1-\delta\to \| x-y\| <\eps\r),
\end{equation*}
{\bf a} modulus of uniform convexity; see Section~\ref{sec:pre}. This is closer to the spirit of computable 
analysis, where $\eps,\delta>0$ are taken as dyadic rational numbers $2^{-k},2^{-n}$ and then 
moduli are number-theoretic functions $\delta_X\col\nat\to\nat$ providing an explicit numerical 
witness for the positivity of an optimal modulus. Such `nonoptimal' moduli are usually easy to 
compute whereas the optimal ones might not be computable. Also e.g. the nonoptimal (but asymptotically 
optimal) modulus of uniform convexity $\frac{\eps^2}8$ for Hilbert spaces 
has a better multiplicative behavior w.r.t. $\eps$ than 
the optimal one. Although, when using arbitrary moduli, one in general no longer can rely on properties of the 
optimal modulus such as monotonicity, this does not cause a real problem as one can use 
instead the monotonicity of the property $A(\eps,\delta):$
\begin{equation*}
\eps_1\le \eps_2\wedge \delta_1\ge \delta_2 \wedge A(\eps_1,\delta_1)\to A(\eps_2,\delta_2).
\end{equation*}
We sometimes use moduli $\delta_X(\eps)$ also in contexts where $\eps>0$ is not restricted to $(0,2].$ Note that 
for $\eps>2$, the property $A(\eps,\delta)$ trivially holds for any $\delta$ and so we can arbitrarily extend 
the modulus to all $\eps>0$ and in the case of moduli such as $\frac{\eps^2}8$ which are already defined for 
all $\eps>0$ we can just take this value also for $\eps>2.$  

Let us now briefly recall some negative results outside of uniformly convex Banach spaces. If~$C$ is a~convex closed subset of a~reflexive strictly convex Banach space~$X,$ the metric projection $P_C\col X\to C$ is a well-defined single-valued mapping, which however is not necessarily continuous. A counterexample with $C$ being a codimension~$2$ subspace is due to Brown~\cite{brown}. In \cite[p. 12]{gr}, Goebel and Reich refer also to an unpublished counterexample of Kripke.

We would like to mention an alternative definition of a proximal mapping with a Young function. It was introduced by Penot and Ratsimahalo \cite[Definition 3.4]{pr} as follows:
\begin{equation} \label{eq:pmpr} 
 \aprox_{\lam,f}^{\Phi}(x)\as\argmin_{y\in X} \l[f(y)+\lam\Phi\l(\frac{\|x-y\|}{\lam}\r) \r],\quad x\in X,
\end{equation} 
where $\lam>0.$  For $\phi(t)\as t^{p-1},$ where $p\ge2,$ we also recover~\eqref{eq:ppm}. Even though the proximal mapping from~\eqref{eq:pmy} and the one from~\eqref{eq:pmpr} both have similar properties for a \emph{fixed} parameter~$\lam,$ they scale with~$\lam$ differently. It turns out that the scaling is more favorable in definition~\eqref{eq:pmy} in the sense that the properties of that proximal mapping (for instance the variational inequality in~\eqref{eq:varineq2} and uniform continuity on bounded sets in Theorem~\ref{thm:uc}) depend on a modulus of uniform convexity of the function $\Phi\circ\|\cdot\|,$ denoted $\delta_{\Phi\circ\|\cdot\|,r_0}$ which is defined on a ball of radius~$r_0,$ and this radius is \emph{independent} of $\lam\in(0,1].$ On the other hand analogous results for the proximal mapping defined in~\eqref{eq:pmpr} need a modulus $\delta_{\Phi\circ\|\cdot\|,\frac{r}{\lam}}$ with~$\frac{r}{\lam}$ going to infinity as $\lam\to0.$ For such a reason we decided to prefer~\eqref{eq:pmy} to~\eqref{eq:pmpr}; see also Remarks~\ref{rem:varineq} and~\ref{rem:ucon}. On the other hand we stress that the \emph{Moreau envelope} corresponding to~\eqref{eq:pmpr}, that is,
\begin{equation*}
 f_\lam(x)\as\inf_{y\in X} \l[f(y)+\lam\Phi\l(\frac{\|x-y\|}{\lam}\r) \r],\quad x\in X,
\end{equation*}
has a deeper meaning. It is known as the \emph{Hopf-Lax formula} and is related to (viscosity) solutions to the Hamilton-Jacobi equations. Namely, a function
\begin{equation*}
 u(t,x)\as \inf_{y\in \rls^n} \l[h(y)+t\Phi\l(\frac{\|x-y\|}{t}\r) \r],\quad x\in \rls^n, t>0,
\end{equation*}
is, under appropriate assumptions, a (viscosity) solution to the Hamilton-Jacobi equation
\begin{align*}
 \partial_t u(t,x) +\Phi^*(\partial_x u(t,x)) & = 0 \\  u(0,x) & =h(x).
\end{align*}
Here $\Phi^*$ stands for the Fenchel-Legendre transformation of~$\Phi.$ For more details, the interested reader is referred to Evans' book~\cite{evans}. Admittedly, we do know of any deeper meaning of the Moreau envelope corresponding to~\eqref{eq:pmy}.

We conclude this Introduction by mentioning related directions of research. The continuity of metric projections in Banach spaces which are both uniformly convex and uniformly smooth was established for instance in \cite[Theorem 2.8]{bl}. The continuity of proximal mappings as well as a closely related problem of the differentiability of Moreau envelopes have been studied at varying levels of generality (e.g. even for nonconvex functions) by a number of authors including Bernard, Thibault and Zlateva \cite{btz}, Cepedello-Boiso \cite{cepedello-boiso}, Kecis and Thibault \cite{kecis-thibault}, Ngai and Penot \cite{np}, Str\"omberg \cite{stroemberg-ark,stroemberg-dis}. However, all those results rely on the (uniform) smoothness of the norm (in addition to uniform convexity), whereas it is our purpose in the present paper to obtain the continuity of proximal mappings without any differentiability assumptions on the norm. Interestingly, despite of the fact that we do not require the norm to be smooth, we obtain the same H\"older constant in the case of a power type~$p$ uniformly convex norm. Indeed, both \cite[Proposition 4.1]{kecis-thibault} and our Corrolary~\ref{cor:hoelder} give the H\"older constant~$\rec{p}$ in this case. Admittedly, the result in \cite[Proposition 4.1]{kecis-thibault} applies also to nonconvex functions.

Finally, we would like to mention that there exists a rich theory of proximal mappings with Bregman divergences; see for instance \cite[3.1.5]{bi}. Yet another type of proximal mappings $\prox\col X^*\to 2^X$ was studied in~\cite{brs}.
 
\section{Preliminaries on uniformly convex spaces and functions} \label{sec:pre}

Let $\l(X,\|\cdot\|\r)$ be a Banach space. If for each $\eps\in(0,2]$ there exists $\delta_X(\eps)>0$ such that
\begin{equation*}
 \delta_X(\eps)\le 1-\l\|\hlf{x+y}\r\|,
\end{equation*}
for every $x,y\in B_X$ satisfying $\|x-y\|\ge\eps,$ we say that $\l(X,\|\cdot\|\r)$ is \emph{uniformly convex} and we call any such function $\delta_X\col(0,2]\to(0,1]$ a \emph{modulus of uniform convexity.}  If there exist $K>0$ and $p\ge2$ such that $\delta_X(\eps)\as K\eps^p,$ for every $\eps\in(0,2],$ is a modulus of uniform convexity, we say that~$X$ has~a modulus of uniform convexity of \emph{power type}~$p.$

There are several fundamental renorming theorems related to uniform convexity. Enflo's~\cite{enflo} and James' \cite{james,james-super} theorems together give that a Banach space admits an equivalent uniformly convex norm if and only if it is superreflexive. 

Enflo \cite{enflo} also showed that a Banach space admits an equivalent uniformly convex norm if and only if it admits an equivalent uniformly smooth norm. Combined with a result of Asplund~\cite{asplund-av}, it implies that a Banach space which admits a uniformly convex renorming admits an equivalent norm which is both uniformly convex and uniformly smooth.

A theorem of Pisier~\cite{pisier} says that each uniformly convex Banach space admits an equivalent norm with a~modulus of uniform convexity of power type $p,$ for some $p\ge2.$

For uniform convexity in metric spaces, see for instance a recent paper by Kell~\cite{kell} and the references therein.

To our knowledge, Asplund was the first to define uniform convexity for functions~\cite{asplund-av}. A convex lsc function $h\fun$ is \emph{uniformly convex} on a convex set $C\subset X$ if for each $\eps>0$ there exists $\delta_{h,C}(\eps)>0$ such that
\begin{equation} \label{eq:ucfunmodulus}
 \delta_{h,C}(\eps)\le \half h(x)+\half h(y)-h\l(\hlf{x+y}\r),
\end{equation}
for every $x, y\in C\cap\dom h$ with $\|x-y\|\ge\eps.$ Here $\dom h\as\l\{x\in X\col h(x)<\infty \r\}$ stands for the \emph{domain} of the function~$h.$ The function $\delta_{h,C}\mdl$ is called a \emph{modulus of uniform convexity.} To simplify the notation, we will write $\delta_{h,r}$ instead of $\delta_{h,B(0,r)}.$ If there exist $K>0$ and $p\ge1$ such that $\delta_{h,C}(\eps)\as K\eps^p,$ for every $\eps\in(0,\infty),$ is a modulus of uniform convexity, we say that~$h$ has a~modulus of uniform convexity of \emph{power type}~$p.$ Equivalently, we can say that $h$ is uniformly convex on~$C$ if, given $\eps>0,$ there exists $\gamma_{h,C}(\eps)>0$ such that we have 
\begin{equation} \label{eq:ucfungauge}
 h\l((1-t)x+ty\r)\le (1-t)h(x)+th(y)-t(1-t)\gamma_{h,C}(\eps),
\end{equation}
for every $t\in[0,1]$ and $x,y\in C\cap\dom h$ with $\|x-y\|\ge\eps.$ Indeed, if we have~\eqref{eq:ucfungauge}, then inequality~\eqref{eq:ucfunmodulus} holds true with $\delta_{h,C}(\eps)\as\frac14\gamma_{h,C}(\eps).$ On the other hand, if we have~\eqref{eq:ucfunmodulus}, then one can put $\gamma_{h,C}(\eps)\as2\delta_{h,C}(\eps)$ to obtain~\eqref{eq:ucfungauge}; see \cite[Remark 2.1]{zalinescu83} or \cite[p. 203]{zalinescu}.

We shall need the following result due to Z{\u{a}}linescu~\cite{zalinescu83}.
\begin{thm}[Z{\u{a}}linescu] \label{thm:uctoum}
 Let $h\fun$ be a convex lsc function and $C\subset X$ be a convex set. If~$h$ satisfies~\eqref{eq:ucfungauge}, then  we have
 \begin{equation*}
  \lang x^*-y^*,x-y\rang \ge 2\gamma_{h,C}(\eps).
 \end{equation*}
for every $x,y\in C\cap\dom h$ with $\|x-y\|\ge\eps$ and $x^*\in \partial h(x)$ and $y^*\in \partial h(y).$
\end{thm}
\begin{proof}
 See \cite[Theorem 2.2]{zalinescu83} or \cite[Corollary 3.4.4]{zalinescu}.
\end{proof}

\section{Results}

Throughout this section we assume that $(X,\|\cdot\|)$ is a Banach space and $f\fun$ is a convex lsc function. If $X$ is uniformly convex, we consider a proximal mapping $\prox_{\lam,f}^{\Phi}$ defined in~\eqref{eq:pmy}. An important ingredient for our results is the following theorem due to Z{\u{a}}linescu~\cite{zalinescu83}.
\begin{thm}[Z{\u{a}}linescu] \label{thm:zal}
Let $(X,\|\cdot\|)$ be uniformly convex. Then the function $\Phi\circ\|\cdot\|$ is uniformly convex on each bounded subset of~$X.$
\end{thm}
\begin{proof}
 See \cite[Theorem 4.1]{zalinescu83} or \cite[Theorem 3.7.7]{zalinescu}.
\end{proof}
We will now present a quantitative version in which we establish an explicit modulus of uniform convexity of the function~$\Phi\circ\|\cdot\|$ in terms of the modulus of uniform convexity of~$X$ and of the properties of the Young function~$\Phi.$ To this end, we introduce the following notation. Let $r>0.$ Since $\Phi$ is strictly increasing and continuous, there exists $\xi_{\Phi,r}(\eps)>0$ satisfying
\begin{equation*}
 \Phi(\beta)\ge\Phi(\alpha)+\xi_{\Phi,r}(\eps).
\end{equation*}
whenever $\alpha,\beta\in[0,r]$ and $\eps>0$ with $\beta\ge \alpha+\eps.$ Since $\Phi$ is continuous, there exists $\omega_{\Phi,r}(\eps)>0$ such that 
\begin{equation*}
 |\alpha-\beta|<\omega_{\Phi,r}(\eps) \implies \l|\Phi(\alpha)-\Phi(\beta)\r|<\eps,
\end{equation*}
whenever $\alpha,\beta\in[0,r]$ and $\eps>0.$ Since $\Phi$ is strictly convex, it is uniformly convex on $[0,r],$ and hence there exists a modulus of uniform convexity of~$\Phi$ on this interval~$\delta_{\Phi,[0,r]},$ which we will denote shortly by~$\delta_{\Phi,r}.$
\begin{prop} \label{prop:modulcomp}
 Let $(X,\|\cdot\|)$ be uniformly convex with a modulus $\delta_X$ and let $r>0.$ Then for every $\alpha,\beta\in[0,r]$ and every $x,y\in B_X$ such that $\|\alpha x-\beta y\|\ge\eps$ for some $\eps>0$ we have
 \begin{equation*}
  \Phi\l(\hlf{\|\alpha x+\beta y\|}\r)\le\half\Phi(\alpha)+\half\Phi(\beta)-\delta_r(\eps),
 \end{equation*}
with 
\begin{equation*}
 \delta_r(\eps)\as \min\l\{\delta_{\Phi,r}\l(\tilde{\eps}\r),\breve{\eps} \r\},
\end{equation*}
where we put
\begin{equation*}
 \tilde{\eps}\as\min\l\{\hlf{\eps},\omega_{\Phi,\frac32r}\l(2\breve{\eps}\r)\r\},\quad\text{and}\quad \breve{\eps}\as\frac13\xi_{\Phi,r}\l(\frac{\eps}{4}\delta_X\l(\frac{\eps}{2r} \r)\r).  
\end{equation*}
This $\delta_r$ is then a modulus of uniform convexity of $\Phi\circ\|\cdot\|$ on $B(0,r),$ that is, $\delta_{\Phi\circ\|\cdot\|,r}$ in the notation introduced in Section~\ref{sec:pre}. Note that if $\eps>2r,$ then $\delta_r$ can be defined arbitrarily, since $2r\ge\|\alpha x-\beta y\|\ge\eps.$ 
\end{prop}
\begin{proof}
By contradiction. Assume there exist $\eps>0$ and $x,y\in B_X$ and $\alpha,\beta\in[0,r]$ such that $\|\alpha x-\beta y\|\ge\eps$ and
\begin{equation} \label{eq:0}
 \Phi\l(\hlf{\|\alpha x+\beta y\|}\r)>\half\Phi(\alpha)+\half\Phi(\beta)-\delta_r(\eps).
\end{equation}
Without loss of generality we may assume $\alpha\le\beta.$ Then
\begin{equation} \label{eq:1}
 \Phi\l(\hlf{\alpha+\beta} \r)\ge\Phi\l(\hlf{\|\alpha x+\beta y\|}\r)>\half\Phi(\alpha)+\half\Phi(\beta)-\delta_r(\eps)\ge \half\Phi(\alpha)+\half\Phi(\beta)-  \delta_{\Phi,r}\l(\tilde{\eps}\r) ,
\end{equation}
and we arrive at
\begin{equation} \label{eq:2}
 |\alpha-\beta|<\tilde{\eps}.
\end{equation}
Moreover,
\begin{equation} \label{eq:34}
 \eps\le\|\alpha x-\beta y\|\le \alpha \|x-y\|+|\alpha-\beta|<\min\l\{ 2\alpha+\tilde{\eps}, r\|x-y\|+\tilde{\eps}\r\}.
\end{equation}
By \eqref{eq:34} we have 
\begin{equation}\label{eq:5}
 \alpha>\hlf{\eps-\tilde{\eps}}\ge\frac{\eps}{4}.
\end{equation}
Then by \eqref{eq:2}
\begin{equation} \label{eq:6}
 \l\|\hlf{\alpha x+\beta y}\r\|\le\alpha\hlf{\|x+y\|}+\half|\alpha-\beta|<\alpha\hlf{\|x+y\|}+\hlf{\tilde{\eps}}.
\end{equation}
By \eqref{eq:34} we also obtain
\begin{equation} \label{eq:7}
 \|x-y\|>\frac{\eps-\tilde{\eps}}{r}\ge\frac{\eps}{2r}.
\end{equation}
Furthermore, we get 
\begin{align*}
 \Phi(\alpha) -\delta_r(\eps) & < \Phi \l(\hlf{\|\alpha x+\beta y\|} \r), \qquad\text{by \eqref{eq:0},} \\ &\le \Phi\l(\alpha\l\| \hlf{x+y}\r\|+\hlf{\tilde{\eps}}\r), \qquad\text{by \eqref{eq:6} and }\Phi\text{ being increasing}, \\ &\le \Phi\l(\alpha\l\|\hlf{x+y} \r\|\r) + 2\breve{\eps}, \qquad\text{since } \alpha\l\| \hlf{x+y}\r\|+\hlf{\tilde{\eps}}<r+\frac{\eps}4\text{ and } \tilde{\eps}\le \omega_{\Phi,\frac32r}\l(2\breve{\eps}\r).
\end{align*}
Consequently,
\begin{equation} \label{eq:9}
 \Phi(\alpha) -\delta_r(\eps) -2\breve{\eps}<\Phi\l(\alpha\l\|\frac{x+y}2\r\| \r).
\end{equation}
Now suppose
\begin{equation}\label{eq:10}
 \l\|\hlf{x+y} \r\|\le 1-\delta_X\l( \frac{\eps}{2r}\r).
\end{equation}
Then
\begin{equation*}
 \alpha-\alpha\l\| \hlf{x+y}\r\|=\alpha\l(1-\l\| \hlf{x+y}\r\| \r)>\frac{\eps}{4}\delta_X\l( \frac{\eps}{2r}\r),
\end{equation*}
where we used~\eqref{eq:5}. Hence,
\begin{equation*}
 \Phi(\alpha) \ge \Phi\l(\alpha \l\| \hlf{x+y}\r\|\r) +\xi_{\Phi,r}\l(\frac{\eps}{4}\delta_X\l(\frac{\eps}{2r} \r)\r).
\end{equation*}
Together with \eqref{eq:9} this gives
\begin{equation*}
 3\breve{\eps}\ge\delta_r(\eps)+2\breve{\eps}>\xi_{\Phi,r}\l(\frac{\eps}{4}\delta_X\l(\frac{\eps}{2r} \r)\r)=3\breve{\eps},
\end{equation*}
which is a contradiction. Therefore \eqref{eq:10} is false. Hence
\begin{equation*}
 \l\|\hlf{x+y} \r\|> 1-\delta_X\l( \frac{\eps}{2r}\r),
\end{equation*}
which implies $\|x-y\|<\frac{\eps}{2r},$ but this is impossible on account of~\eqref{eq:7}. The proof is hence complete.
\end{proof}

Let us proceed by showing basic properties of the proximal mapping. The result in Proposition \ref{prop:goodvarineq} as well as its variant in Proposition \ref{prop:welldef} follow by standard arguments.
\begin{prop} \label{prop:goodvarineq}
Let $(X,\|\cdot\|)$ be a uniformly convex Banach space. The proximal mapping from~\eqref{eq:pmy} is well defined and single-valued. Given $x\in X$ and $\lam>0,$ denote $x_\lam\as \prox_{\lam,f}^{\Phi}(x).$ For $\lam_0\as 1,$ choose $r_0>\l\|x_{\lam_0}-x\r\|.$ Then we have
 \begin{equation} \label{eq:varineq2}
 f(y)+\rec{\phi(\lam)}\Phi\l(\|x-y\|\r) \ge f\l(x_\lam\r)+\rec{\phi(\lam)}\Phi\l(\l\|x-x_{\lam}\r\|\r)+\frac{2}{\phi(\lam)}\delta_{\Phi\circ\|\cdot\|,r_0}\l(\l\|y-x_\lam\r\| \r),
 \end{equation}
for every $y\in B\l(x,r_0\r)$ and $\lam\in(0,1].$
\end{prop}
\begin{proof}
Let $\lam>0$ and denote
 \begin{equation*}
  h\as f+\rec{\phi(\lam)}\Phi\l(\|x-\cdot\|\r).
 \end{equation*}
The uniform convexity of~$h$ on bounded sets along with its coercivity imply that $\prox_{\lam,f}^{\Phi}\map$ is a well-defined single-valued mapping.

We now claim that $\lam\mapsto\l\|x_{\lam}-x\r\|$ is nondecreasing. Indeed, if $0<\kappa<\lam,$ then
\begin{equation*} f\l(x_\lam\r)+\frac1{\phi(\kappa)}\Phi\l(\l\|x-x_\lam\r\|\r)\ge f\l(x_\kappa\r)+\frac1{\phi(\kappa)}\Phi\l(\l\|x-x_\kappa\r\|\r),\end{equation*}
and furthermore,
\begin{align*}
f\l(x_\lam\r)+\frac1{\phi(\lam)}\Phi\l(\l\|x-x_\lam\r\|\r) & \ge f\l(x_\kappa\r)+\frac1{\phi(\lam)}\Phi\l(\l\|x-x_\kappa\r\|\r)  \\ & \quad +\l(\frac1{\phi(\kappa)}-\frac1{\phi(\lam)}\r)\l[\Phi\l(\l\|x-x_\kappa\r\|\r)-\Phi\l(\l\|x-x_\lam\r\|\r)\r],
\end{align*}
which already implies $\l\|x-x_\kappa\r\|\le \l\|x-x_\lam\r\|.$

Note that the function~$h$ is uniformly convex on~$B\l(x,r_0\r),$ more precisely,
\begin{equation}\label{eq:hisuc}
 h\l(\hlf{u+v}\r)\le\half h(u)+\half h(v)-\rec{\phi(\lam)}\delta_{\Phi\circ\|\cdot\|,r_0}\l(\|u-v\|\r),
\end{equation}
for every $u,v\in B\l(x,r_0\r).$ 

Since for $\lam\in(0,1]$ we have $r_0>\l\|x_{\lam_0}-x\r\|\ge \l\|x_\lam-x\r\|,$ we can apply inequality~\eqref{eq:hisuc} with $u\as x_\lam$ and $v\as y,$ for an arbitrary $y\in B\l(x,r_0\r),$ to arrive at
\begin{equation*}
 \frac{2}{\phi(\lam)}\delta_{\Phi\circ\|\cdot\|,r_0}\l(\l\|x_\lam-y\r\|\r) \le h\l(x_\lam\r)+h(y)-2h\l(\hlf{x_\lam+y} \r)  \le h\l(x_\lam\r)+h(y)-2h\l(x_\lam\r)  \le h(y)-h\l(x_\lam\r).
\end{equation*}
This gives~\eqref{eq:varineq2}.
\end{proof}

\begin{rem} \label{rem:varineq}
Note that in the variational inequality in~\eqref{eq:varineq2}, the last term on the right-hand side depends on the modulus of uniform convexity $\delta_{\Phi\circ\|\cdot\|,r_0},$ which is defined on a ball with radius $r_0$ and, importantly, this $r_0$ is \emph{independent} of $\lam\in(0,1].$ This fact is in sharp contrast with a variational inequality in~\eqref{eq:varineq} below for the proximal mapping~\eqref{eq:pmpr}, in which the modulus of uniform convexity $\delta_{\Phi\circ\|\cdot\|,\frac{r}{\lam}}$ is defined on a~ball with radius $\frac{r}{\lam},$ which, of course, grows to infinity as $\lam\to0.$ That is, as we decrease $\lam,$ we need to use a new modulus which is defined on a bigger ball.

One reason why we are interested in the range $\lam\in\l(0,\lam_0\r]$ for some $\lam_0>0,$ say $\lam_0\as 1,$ is that iterative applications of proximal mappings/resolvents with decreasing step sizes (e.g. $\lam_n\as\frac{t}{n},$ for a fixed time $t$ and every $n\in\nat$) lead to solutions to abstract Cauchy problems, see for instance \cite{ags,bp70isr,bp70jfa,cl,hille,reich2,reich3,reich1}.
\end{rem}

Let us now hence present a variational inequality for the proximal mapping~\eqref{eq:pmpr}.
\begin{prop} \label{prop:welldef}
 Let $(X,\|\cdot\|)$ be a uniformly convex Banach space and $\lam>0.$ Then the proximal mapping $\aprox_{\lam,f}^{\Phi}$ from~\eqref{eq:pmpr} is well defined and single-valued. Moreover, if for $x\in X$ we denote $x_\lam\as\aprox_{\lam,f}^{\Phi}(x),$ then for $r>0$ such that $ r\ge \l\|x-x_\lam\r\|,$ we have
 \begin{equation} \label{eq:varineq}
  f(y)+\lam\Phi\l(\frac{\|x-y\|}\lam\r) \ge f\l(x_\lam\r)+\lam\Phi\l(\frac{\l\|x-x_\lam\r\|}{\lam}\r)+2\lam\delta_{\Phi\circ\|\cdot\|,\frac{r}{\lam}}\l(\l\|\frac{y-x_\lam}{\lam}\r\| \r),
 \end{equation}
for every $y\in B\l(x, r\r).$ 
\end{prop}
\begin{proof}
Choose $x\in X$ and denote $h\as f+\lam\Phi\l(\frac{\|x-\cdot\|}{\lam}\r).$ Note that the function~$h$ is uniformly convex on~$B\l(x, r\r),$ more precisely,
\begin{equation} \label{eq:uch}
 h\l(\hlf{u+v}\r)\le\half h(u)+\half h(v)-\lam\delta_{\Phi\circ\|\cdot\|,\frac{r}{\lam}}\l(\frac{\|u-v\|}{\lam}\r)
\end{equation}
for every $u,v\in B\l(x, r\r).$ The fact that the proximal mapping is well defined and single-valued is a~consequence of the uniform convexity of~$h$ on bounded sets (Theorem~\ref{thm:zal}) and coercivity.

Applying inequality~\eqref{eq:uch} with $u\as x_\lam$ and $v\as y,$ for an arbitrary $y\in B\l(x, r\r),$ yields
\begin{equation*}
 2\lam\delta_{\Phi\circ\|\cdot\|,\frac{r}{\lam}}\l(\frac{\l\|x_\lam-y\r\|}{\lam}\r) \le  h\l(x_\lam\r)+h(y)-2h\l(\hlf{x_\lam+y} \r) \le h\l(x_\lam\r)+h(y)-2h\l(x_\lam\r) \le h(y)-h\l(x_\lam\r).
\end{equation*}
This gives~\eqref{eq:varineq}.
\end{proof}

To obtain further properties of the proximal mapping, we introduce a~function $\eta_\Phi\col(0,\infty)\to(0,\infty)$ such that 
\begin{equation} \label{eq:decreasezero}
 \frac{\Phi\l(\eta_{\Phi}(t)\r)}{\eta_{\Phi}(t)} \le t,
\end{equation}
for every $t\in(0,\infty).$ This function is to witness the first property of~$\Phi$ in~\eqref{eq:young}. (Note that $s\mapsto\frac{\Phi(s)}{s}$ is strictly increasing.) Observe that it also witnesses the limit behavior of $\phi$ at~$0.$ Indeed,
\begin{equation*}
 \phi\l( \hlf{\eta_{\Phi}(t)} \r)\le \frac{\Phi\l(\eta_{\Phi}(t) \r)-\Phi\l(\hlf{\eta_{\Phi}(t)} \r)}{\eta_{\Phi}(t)-\hlf{\eta_{\Phi}(t)}}\le2\frac{\Phi\l( \eta_{\Phi}(t) \r)}{\eta_{\Phi}(t)}\le 2t,
\end{equation*}
for every $t\in(0,\infty).$ In a similar way we quantify the second property of~$\Phi$ in~\eqref{eq:young}, that is, we introduce a~function $\rho_\Phi\col(0,\infty)\to(0,\infty)$ such that 
\begin{equation} \label{eq:ratedivergence}
 \frac{\Phi\l(\rho_{\Phi}(t)\r)}{\rho_{\Phi}(t)} \ge t,
\end{equation}
for every $t\in(0,\infty).$

\begin{prop} \label{prop:lamtozero}
 Let $\l(X,\|\cdot\|\r)$ be uniformly convex with a modulus of uniform convexity~$\delta_X\le1.$ Given $x\in X$ and $\lam>0,$ we denote $x_\lam\as \prox_{\lam,f}^{\Phi}(x).$ Then $x_\lam\to P_{\cldom f}(x)$ as $\lam\to0.$ More quantitatively, given $\eps>0,$ we have
\begin{equation*}
\l\|x_\lam- P_{\cldom f}(x)\r\|<\eps,
\end{equation*}
for every $\lam\in(0,\Lambda),$ where
\begin{equation} \label{eq:biglam}
 \Lambda\as\min\l\{1,\half\eta_\Phi\l(\frac{\hlf{\tilde{\eps}}\phi\l(\frac{\eps}{5}\r)}{\zeta} \r) \r\},
\end{equation}
and $\tilde{\eps}\as \frac{\eps}{10}\delta_X\l(\frac{\eps}{\beta}\r)$ with $\alpha\as f\l(x_{\lam_0}\r)$ and $\beta\as \l\|x-x_{\lam_0}\r\|$ for $\lam_0\as1.$ And $\zeta>0$ is a constant such that $\zeta>f(z)-\alpha$ for some
\begin{equation*}
  z\in B\l(P_{\cldom f}(x), \min\l\{\frac{\eps}{20},\tilde{\eps}\r\} \r).
\end{equation*}
\end{prop}
\begin{proof}
Since $\lam\mapsto\l\|x-x_\lam\r\|$ is nondecreasing and since
\begin{equation*}
f\l(x_\lam\r)=\inf_{y\in X}\l\{f(y)\col \|x-y\|\leq \l\|x-x_\lam\r\|\r\},
\end{equation*}
we obtain that $\lam\mapsto f\l(x_\lam\r)$ is nonincreasing.

For $\lam_0\as1$ set $\alpha\as f\l(x_{\lam_0}\r)$ and $\beta\as \l\|x-x_{\lam_0}\r\|.$ Denote $p\as P_{\cldom f}(x).$ Assume $\l\|p-x_\lam\r\|\ge\eps$ for some $\eps>0$ and $\lam>0.$ Let us consider two cases.

Case 1: Assume $\|x-p\|\le\frac{\eps}5.$ Choose $z\in\dom f$ such that $\|p-z\|\le\frac{\eps}{20}$ and $\zeta>0$ with $\zeta>f(z)-\alpha.$ Then
\begin{equation*}
   f\l(x_\lam\r)+\rec{\phi(\lam)}\Phi\l(\l\| x-x_\lam\r\|\r)\le f(z)+\rec{\phi(\lam)}\Phi\l(\l\| x-z\r\|\r)\le f(z)+\rec{\phi(\lam)}\Phi\l(\frac{\eps}{4} \r),
\end{equation*}
and therefore, for $\lam\in(0,1],$ we have
\begin{equation*}
 \frac{4\eps}{5}\phi\l(\frac{\eps}{4} \r)  \le \Phi\l(\l\| x-x_\lam\r\|\r)-\Phi\l(\frac{\eps}{4} \r)\le \phi(\lam)\l(f(z)-f\l(x_\lam\r)\r)\le  \phi(\lam)\l(f(z)-\alpha\r) \le  \phi(\lam)\zeta,
\end{equation*}
since $\l\| x-x_\lam\r\|\ge\l\|x_\lam - p\r\|-\|x-p\|\ge \eps-\frac{\eps}{5}\ge\frac{4\eps}{5}.$ Hence for $\lam$ sufficiently small, namely,
\begin{equation} \label{eq:lam1}
 \lam<\half\eta_\Phi\l( \frac{ \frac{2\eps}{5}\phi \l( \frac{\eps}{4}\r)}{\zeta}    \r),
\end{equation}
we get a contradiction.

Case 2: Assume $\|x-p\|>\frac{\eps}{5}.$ By the uniform convexity of~$X$ we have
\begin{equation} \label{eq:ucproj}
\|x-p\|+\|x-p\|\delta_X\l(\frac{\eps}{\beta}\r)\le \l\|x-x_\lam\r\|.
\end{equation}
Indeed, since $\|x-p\|\le \l\|x-x_\lam\r\|\le\beta,$ we have
\begin{equation*}
 \frac{x-p}{\l\|x-x_\lam\r\|},\frac{x-x_\lam}{\l\|x-x_\lam\r\|}\in B_X,\qquad\text{and}\qquad \rec{\l\|x-x_\lam\r\|} \l\|p-x_\lam\r\|\ge \rec{\beta}\l\|p-x_\lam\r\|\ge\frac{\eps}{\beta},
\end{equation*}
and the definition of uniform convexity yields
\begin{equation*}
 1-\l\|\frac{x-p+x-x_\lam}{2\l\|x-x_\lam\r\|}\r\|\ge\delta_X\l(\frac{\eps}{\beta} \r),
\end{equation*}
and consequently,
\begin{equation*}
 \l\|x-x_\lam\r\|-\l\|p-x\r\|\ge\l\|x-x_\lam\r\|\delta_X\l(\frac{\eps}{\beta} \r),
\end{equation*}
which implies the desired inequality in~\eqref{eq:ucproj}.

Next set
\begin{equation*}
 \tilde{\eps}\as \frac{\eps}{10}\delta_X\l(\frac{\eps}{\beta}\r).
\end{equation*}
By the monotonicity of $\Phi$ and by~\eqref{eq:ucproj} we obtain
\begin{equation} \label{eq:estim1}
f\l(x_\lam\r)+\rec{\phi(\lam)}\Phi\l(\|x-p\|+2\tilde{\eps}\r) \le f\l(x_\lam\r)+\rec{\phi(\lam)}\Phi\l(\l\|x-x_\lam\r\|\r)\le f(z)+\rec{\phi(\lam)}\Phi\l(\l\| x-z\r\|\r),
\end{equation}
for each $z\in X.$ Choose $z\in\dom f$ such that $\|p-z\|\le\tilde{\eps}$ and $\zeta>0$ with $\zeta>f(z)-\alpha.$ Then~\eqref{eq:estim1} yields
\begin{equation*}
f\l(x_\lam\r)+\rec{\phi(\lam)}\Phi\l(\|x-p\|+2\tilde{\eps}\r) \le f(z)+\rec{\phi(\lam)}\Phi\l(\l\|x-p\r\|+\tilde{\eps}\r), 
\end{equation*}
and hence, for $\lam\in(0,1],$ we have
\begin{align*}
 \phi\l(\frac{\eps}{5}\r) \tilde{\eps} & \le\phi\l(\l\| x-p\r\|\r) \tilde{\eps}\le\Phi\l(\|x-p\|+2\tilde{\eps}\r)-\Phi\l(\l\| x-p\r\| +\tilde{\eps}\r) \\ & \le \phi(\lam)\l(f(z)-f\l(x_\lam\r)\r)  \\ & \le \phi(\lam)\l(f(z)-\alpha\r) \\ & \le \phi(\lam)\zeta. 
\end{align*}
Therefore, for $\lam$ satisfying
\begin{equation} \label{eq:lam2}
 \lam<\half\eta_\Phi\l(\frac{\hlf{\tilde{\eps}}\phi\l(\frac{\eps}{5}\r)}{\zeta} \r)
\end{equation}
we obtain a contradiction.

Comparing the condition in~\eqref{eq:lam1} with the one in~\eqref{eq:lam2} we see that~\eqref{eq:lam2} is more restrictive and hence we can set $\Lambda$ as in~\eqref{eq:biglam} to complete the proof.
\end{proof}

As a consequence of the lower semicontinuity of~$f,$ we obtain the corresponding limit behavior of function values.
\begin{cor}
 If $x\in\cldom f$ and we denote $x_\lam\as\prox_f^{\Phi, \lam}(x),$ then $f\l(x_\lam \r)\to f(x)$ as $\lam\to0.$
\end{cor}
\begin{proof}
By virtue of Proposition~\ref{prop:lamtozero}, we know that $x_\lam\to x$ as $\lam\to0.$ Therefore combining the fact $f(x)\ge f\l(x_\lam\r)$ with the lower semicontinuity of~$f,$ that is,
 \begin{equation*}
  \liminf_{\lam\to0} f\l(x_\lam\r)\ge f(x),
 \end{equation*}
gives the desired result.
\end{proof}

Let us continue by providing a useful characterization of the proximal mapping. It is an easy consequence of a result of Asplund.
\begin{lem} \label{lem:proxchar}
 Let $(X,\|\cdot\|)$ be uniformly convex and $\lam>0.$ Given $u,\ol{u}\in X,$ we have $\ol{u}=\prox_{\lam,f}^{\Phi}(u)$ if and only if there exists $u^*\in \rec{\phi(\lam)}J_\phi\l(u-\ol{u}\r)$ such that
 \begin{equation*} 
  f(v)-f\l(\ol{u}\r)\ge\lang u^*,v-\ol{u}\rang,
 \end{equation*}
for each $v\in X.$ 
\end{lem}
\begin{proof}
 Since $\ol{u}=\prox(u)$ is equivalent to 
\begin{equation*}
 0\in\partial \l(f+\rec{\phi(\lam)}\Phi\l(\|u-\cdot\|\r)\r)\l(\ol{u}\r),
\end{equation*}
it is also equivalent to the existence of $u^*\in X^*$ such that $u^*\in\partial f\l(\ol{u}\r)$ and $-u^*\in\partial\l(\rec{\phi(\lam)}\Phi\circ\|u-\cdot\| \r)\l(\ol{u}\r).$ The latter inclusion then reads
\begin{equation*}
 u^*\in\rec{\phi(\lam)}\partial\l(\Phi\circ\|\cdot\| \r)\l(u-\ol{u}\r)=\rec{\phi(\lam)}J_\phi\l(u-\ol{u}\r),
\end{equation*}
where the last equality follows from \cite[Theorem 1]{asplund}.
\end{proof}

The following lemma, which relies on standard arguments and Lemma~\ref{lem:proxchar}, will be needed in the proof of our main result (Theorem~\ref{thm:uc}).
\begin{lem} \label{lem:estim}
Let $(X,\|\cdot\|)$ be uniformly convex, $x\in X$ and $r>0.$ Then there exists $R>0$ such that for every $y\in B(x,r)$ and $\lam\in(0,1]$ we have $\l\|y-\prox_{\lam,f}^{\Phi}(y)\r\|\le R.$
\end{lem}
\begin{proof}
Choose $y\in B(x,r)$ and $\lam\in(0,1]$ and denote $x_\lam\as\prox_{\lam,f}^{\Phi}(x)$ as well as $y_\lam\as\prox_{\lam,f}^{\Phi}(y).$ Set $\lam_0\as 1.$ Since $\l\|y-y_\lam\r\|\le \l\|y-y_{\lam_0}\r\|$ for each $\lam\in(0,1],$ we are concerned with $\lam_0$ only. Let $R_0\as\l\|x-x_{\lam_0}\r\|$ and $R\as\l\|y-y_{\lam_0}\r\|.$ Lemma~\ref{lem:proxchar} provides us with $x^*\in\partial f\l(x_{\lam_0} \r)$ such that $\l\|x^*\r\|=\frac{\phi\l(R_0\r)}{\phi\l(\lam_0\r)}.$ Then we can estimate
\begin{align*}
 f\l(y_{\lam_0} \r)+\rec{\phi\l(\lam_0\r)}\Phi\l(\l\|y-y_{\lam_0}\r\|\r) &\le f\l(x_{\lam_0} \r)+\rec{\phi\l(\lam_0\r)}\Phi\l(\l\|y-x_{\lam_0}\r\|\r) \\ & \le f\l(y_{\lam_0}\r)+\lang -x^*,y_{\lam_0}-x_{\lam_0} \rang+\rec{\phi\l(\lam_0\r)}\Phi\l(\l\|y-x_{\lam_0} \r\| \r) \\ & \le f\l( y_{\lam_0}\r)+\frac{\phi\l(R_0\r)}{\phi\l(\lam_0\r)}\l(R+r+R_0\r)+\rec{\phi\l(\lam_0\r)}\Phi\l(r+R_0\r).
\end{align*}
Consequently,
\begin{equation} \label{eq:optk1}
 \Phi\l(R\r)  \le R \phi\l(R_0\r)+\l(r+R_0\r)\phi\l(R_0\r)+\Phi\l(r+R_0\r).
\end{equation}
Since the left-hand side has superlinear growth in~$R$ and the right-hand side only linear, there exists a~maximal $R$ for which inequality~\eqref{eq:optk1} holds. To obtain a more explicit formula for this~$R,$ we use the function~$\rho_\Phi$ from~\eqref{eq:ratedivergence}, that is, a~function $\rho_\Phi\col(0,\infty)\to(0,\infty)$ such that 
\begin{equation*}
 \frac{\Phi\l(\rho_{\Phi}(t)\r)}{\rho_{\Phi}(t)} \ge t,
\end{equation*}
for every $t\in(0,\infty).$ Then we can conclude from~\eqref{eq:optk1} that $R$ must satisfy
\begin{equation*}
R\le \max\l\{ 1,\rho_\Phi\l(\phi\l(R_0\r)+\l(r+R_0\r)\phi\l(R_0\r)+\Phi\l(r+R_0\r)\r)\r\},
\end{equation*}
and finishes the proof.
\end{proof}

We are now ready to present our main result. It states that the proximal mapping is uniformly continuous on each bounded subset of~$X$ and provides an explicit modulus of uniform continuity. 
\begin{thm}\label{thm:uc}
 Let $\l(X,\|\cdot\|\r)$ be a uniformly convex Banach space and $f$ be a convex lsc function. Then the proximal mapping defined in~\eqref{eq:pmy} is uniformly continuous on bounded subsets of~$X,$ that is, given $z\in X$ and $r>0,$ the following implication holds true for every $\eps>0:$
 \begin{equation} \label{eq:mainimpl}
 \|x-y\|< \delta(\eps) \implies \l\|\prox_{\lam,f}^{\Phi}(x) -\prox_{\lam,f}^{\Phi}(y) \r\| < \eps,
\end{equation}
for every $x,y\in B(z,r)$ and $\lam\in(0,1],$ where
\begin{equation} \label{eq:delta}
\delta(\eps)\as\min\l\{ \hlf{\eps},\frac{2}{\phi\l(R\r)}
\delta_{R}\l(\frac{\eps}{2} \r) \r\},
\end{equation}
and $R>0$ is a constant independent of $\lam\in(0,1]$ and $\delta_{R}$ is a modulus of uniform convexity of $\Phi\circ\|\cdot\|$ on the ball $B\l(0,R\r).$
\end{thm}
\begin{proof}
Choose $z\in X$ and $r>0.$ By virtue of Lemma \ref{lem:estim} there exists an (explicit) $R>0$ such that for every $x\in B(z,r)$ and $\lam\in(0,1]$ we have $\l\|x-\prox_{\lam,f}^{\Phi}(x)\r\|\le R.$  On account of Proposition \ref{prop:modulcomp} we have an explicit modulus of uniform convexity of $\Phi\circ\|\cdot\|$ on the ball $B\l(0,R\r),$ which we denote by $\delta_R.$

Let now $\lam\in(0,1]$ and $x,y\in B(z,r).$ Denote $x_\lam\as\prox_{\lam,f}^{\Phi}(x)$ and $y_\lam\as\prox_{\lam,f}^{\Phi}(y).$ We will proceed by contradiction. Assume that $\l\|x_\lam-y_\lam\r\|\ge\eps$ while $\|x-y\|< \delta(\eps).$ By Lemma \ref{lem:proxchar} we get points $x^*\in \rec{\phi(\lam)}J_\phi\l(x-x_\lam\r)$ and $y^*\in \rec{\phi(\lam)}J_\phi\l(y-y_\lam\r)$ such that 
\begin{align*}
 f\l(y_\lam\r)-f\l(x_\lam\r)\ge\lang x^*,y_\lam-x_\lam\rang,
 \intertext{and,}
 f\l(x_\lam\r)-f\l(y_\lam\r)\ge\lang y^*,x_\lam-y_\lam\rang,
\end{align*}
and after summing up
\begin{equation*}
 0\ge \lang x^*-y^*,y_\lam-x_\lam\rang.
\end{equation*}
Along with 
\begin{equation*}
 \l\|x^*-y^* \r\|\|x-y\| \ge \lang x^*-y^*,x-y\rang,
\end{equation*}
we obtain
\begin{equation*}
 \l\|x^*-y^* \r\|\|x-y\| \ge \lang x^*-y^*,x-x_\lam -y+y_\lam\rang.
\end{equation*}
Since $\|x-y\|<\delta(\eps)\le\hlf{\eps},$ we have $\l\|x-x_\lam -y+y_\lam \r\|\ge\hlf{\eps},$ and Theorem \ref{thm:uctoum} gives
\begin{equation*}
 \l\| x^*-y^* \r\|\|x-y\| \ge \frac{4}{\phi(\lambda)}\delta_{R}\l(\frac{\eps}{2}\r),
\end{equation*}
Since $\l\|x-x_\lam\r\|\le R$ and $\l\|y-y_\lam\r\|\le R,$ we apply the duality mapping and obtain
\begin{equation*}
 \frac{\phi(R)}{\phi(\lam)}\|x-y\| \ge \frac{2}{\phi(\lambda)}\delta_{R}\l(\frac{\eps}{2}\r),
\end{equation*}
which gives a contradiction with the definition of $\delta(\eps).$
\end{proof}

\begin{rem}\label{rem:ucon}
A straightforward modification of the above proof leads to an analogous theorem for the proximal mapping from~\eqref{eq:pmpr}. However, instead of on a modulus $\delta_R$ with $R$ \emph{independent} of $\lam\in(0,1],$ the uniform continuity of $\aprox_{\lam,f}^{\Phi}$ on bounded sets will depend on a modulus $\delta_{\frac{R}{\lam}}$ with $\frac{R}{\lam}$ going to infinity as $\lam\to0.$ More precisely, for the uniform continuity of $\aprox_{\lam,f}^{\Phi}$ we have to replace $\delta(\eps)$ in~\eqref{eq:delta} by
\begin{equation*}
\delta(\eps)\as\min\l\{ \hlf{\eps},\frac{2\lam\delta_{\frac{R}{\lam}}\l(\frac{\eps}{2\lam} \r)}{\phi\l(\frac{R}{\lam}\r)} \r\}.
\end{equation*}
\end{rem}

Theorem \ref{thm:zal} assures the uniform convexity of $\Phi\circ\|\cdot\|$ on bounded sets. We will now ask about uniform convexity on the entire space. Xu~\cite[Theorem 1]{xu} proved the following. 
\begin{thm}[Xu] \label{thm:xu}
 Let $\l(X,\|\cdot \| \r)$ be a Banach space and $p\ge2.$ Then the following are equivalent:
 \begin{enumerate}
  \item The norm $\|\cdot\|$ has modulus of uniform convexity of power type $p.$ \label{i:xu:i}
  \item The function $\|\cdot\|^p$ has modulus of uniform convexity of power type $p.$ \label{i:xu:ii}
  \item The function $\|\cdot\|^p$ is uniformly convex. \label{i:xu:iii}
 \end{enumerate}
\end{thm}
Unaware of this theorem, Ball, Carlen and Lieb proved in \cite[Proposition 7]{bcl} its special case, namely, that $\|\cdot\|$ has a modulus of uniform convexity of power type $p$ if and only if there exists a constant $K>0$ such that
\begin{equation*} 
 \l\|\hlf{x+y}\r\|^p \le \half \|x\|^p + \half\|y\|^p -K\|x-y\|^p,
\end{equation*}
for every $x,y\in X.$ 

In \cite[Theorem  2.3]{bghv}, Theorem \ref{thm:xu} was rediscovered once again and the proof, like the one in~\cite{bcl}, relies on a duality between uniform convexity and uniform smoothness.

In \cite{bv}, the above Theorem \ref{thm:xu} was obtained as a corollary of more general theorems on the uniform convexity of the composition of a norm with a convex function. We first quote \cite[Theorem 2.1]{bv} here.
\begin{thm}[Borwein, Vanderwerff] \label{thm:bv}
 Let $\l(X,\|\cdot \| \r)$ be a Banach space and $\Psi\wgt$ be convex nondecreasing. Then the function $\Psi\circ\|\cdot\|$ is uniformly convex if and only if the function $\Psi$ and the norm $\|\cdot\|$ are both uniformly convex while
 \begin{equation} \label{eq:bv}
  \liminf_{t\to\infty} \Psi_+'(t)\delta_X\l(\frac{\eps}{t}\r) t >0,
 \end{equation}
for each $\eps>0.$
\end{thm}
By inspecting the original proof from~\cite[Theorem 2.1]{bv} (as well as 
using some estimates from the proof of~\cite[Theorem 2.3]{bv}) we will now extract a modulus of uniform convexity for the function $\Psi\circ\|\cdot\|$ in Theorem~\ref{thm:bv}. To this end, we introduce the following notation. Given $\eps>0,$ denote $K_\eps\ge0$ and $\xi_\eps>0$ such that
\begin{equation} \label{eq:witness}
 \Psi_+'(t)\delta_X\l(\frac{\eps}{t}\r)t\ge\xi_\eps,
\end{equation}
for every $t> \max\l\{ K_\eps,\hlf{\eps}\r\}.$ This is to witness~\eqref{eq:bv}.

\begin{prop} \label{prop3.14}
 Let $\l(X,\|\cdot\|\r)$ be a uniformly convex Banach space with a modulus $\delta_X$ such that $\delta_X(\eps)\le \half$ for $\eps\in (0,1],$\footnote{This can always be achieved by just taking the minimum with $\half.$} and $\Psi\wgt$ be an increasing uniformly convex function with a~modulus of uniform convexity~$\delta_\Psi.$ Let~$K_\eps$ and~$\xi_\eps$ be as in~\eqref{eq:witness} above. Then $\Psi\circ\|\cdot\|$ is uniformly convex with a~modulus
 \begin{equation*}
  \delta_{\Psi\circ\|\cdot\|}(\eps)\as\min\l\{\delta_\Psi\l(\hlf{\eps}\r),\frac{\eps}{4}\delta_X\l( \frac{\eps}{\max\l\{
\hlf{\eps},8K_{\frac{\eps}{8}}\r\}}\r)\Psi_+'\l(\frac{\eps}{4}\r), 
\xi_{\frac{\eps}{8}} \r\},
 \end{equation*}
for each $\eps>0.$ 
\end{prop}
\begin{proof}
Let $\eps>0$ and choose $x,y\in X$ such that $\|x-y\|\ge\eps.$ Without loss of generality, suppose $\|x\|\ge\|y\|.$ If $\|x\|-\|y\|\ge\hlf{\eps},$ then 
 \begin{equation*}
  \half\Psi\l(\|x\|\r)+\half\Psi\l(\|y\|\r)-\Psi\l(\l\|\hlf{x+y}\r\|\r)\ge \half\Psi\l(\|x\|\r)+\half\Psi\l(\|y\|\r)-\Psi\l(\hlf{\|x\|+\|y\|}\r)\ge\delta_\Psi\l(\hlf{\eps}\r),
 \end{equation*}
by virtue of the fact that $\Psi$ is increasing.

Let us therefore assume $\|x\|-\|y\|<\hlf{\eps}.$ First observe that
\begin{equation}\label{eq:xny}
\|x\|\ge\hlf{\eps}, \qquad\text{and}\qquad \|y\| \ge \frac{\eps}{4}. 
\end{equation}
Indeed, if it was the case that $\|x\|<\hlf{\eps}$ or $\|y\| < \frac{\eps}{4},$ we would get a contradiction from 
\begin{equation*}
\| x-y\| \le \| x\| +\| y\| \le \left\{
\begin{array}{ll}  2\| x\| <\eps, &\text{if } \| x\| <\hlf{\eps},\\
 \| y\| +\frac{\eps}{2} +\| y\| <\eps, & \text{if } \| y\| <\frac{\eps}4.
 \end{array} \right. 
\end{equation*}
Hence~\eqref{eq:xny} holds true. We now distinguish two cases.

Case 1: Assume $\| y\|\le 2 K_{\frac{\eps}{8}}.$ 
Define  $\tilde{x}\as\frac{x}{\|x\|}$ and $\tilde{y}\as\frac{y}{\|y\|}.$ Then, 
reasoning as in the proof of \cite[Theorem 2.3]{bv}, we obtain $\| y-\| y\| \cdot
\tilde{x}\|>\frac{\eps}{2}$ and so
\begin{equation*}
\| \tilde{y}-\tilde{x}\|>\frac{\eps}{2\| y\|} \ge \frac{\eps}{4K_{\frac{\eps}{8}}}> \frac{\eps}{8K_{\frac{\eps}{8}}}. 
\end{equation*}
By the uniform convexity of $X$ applied to $\tilde{x},\tilde{y},$ we get that
\begin{equation*}
\l\|\hlf{x+y}\r\| \le \|y\| \left\|\hlf{\tilde{x}+\tilde{y}}\right\| +\hlf{\|x\|-\|y\|}\le \half \|x\|+\half \|y\|-\|y\| \cdot \delta_X\l(\frac{\eps}{8K_{\frac{\eps}{8}}}\r).
\end{equation*}
Since $\Psi$ is convex and increasing we get
\begin{align*}
 \half\Psi\l(\|x\|\r)+\half\Psi\l(\|y\|\r)-\Psi\l(\l\|\hlf{x+y}\r\|\r) & \ge \Psi\l(\hlf{\|x\|+\|y\|}\r)-\Psi\l(\l\|\hlf{x+y}\r\|\r) \\ & \ge \Psi\l(\hlf{\|x\|+\| y\|} \r)-\Psi\l(\hlf{\|x\|+\|y\|}-\|y\| \delta_X\l(\frac{\eps}{8K_{\frac{\eps}{8}}}\r) \r)  \\ & \ge \frac\eps4 \delta_X\l(\frac{\eps}{8K_{\frac{\eps}{8}}}\r)\Psi_+'\l(\frac\eps4\r),
\end{align*} 
where we used~\eqref{eq:xny} to obtain the last inequality. Since $\frac{\eps}{8K_{\frac{\eps}{8}}}\le 1$ and therefore $\delta_X\l(\frac{\eps}{8K_{\frac{\eps}{8}}}\r)\le \half,$ we obtain 
\begin{equation*}
\hlf{\|x\| +\|y\|} -\|y\| \delta_X\l(\frac{\eps}{8K_{\frac{\eps}{8}}} \r)\ge \hlf{\|x\| +\|y\|}-\hlf{\|y\|}\ge\frac{\eps}{4},
\end{equation*}
as well as
\begin{equation*}
 \|y\|\delta_X\l(\frac{\eps}{8K_{\frac{\eps}{8}}}\r) \ge \frac{\eps}{4}\delta_X\l(\frac{\eps}{8K_{\frac{\eps}{8}}} \r).
\end{equation*}

Case 2: Assume $\|y\|>2 K_{\frac{\eps}{8}}.$ Reasoning like in Case 1 and using the fact that
\begin{equation*}
\hlf{\|x\| +\|y\|} -\|y\| \delta_X\l(\frac{\eps}{4\|y\|} \r)\ge \hlf{\|x\|}\ge \hlf{\| y\|},
 \end{equation*}
we arrive at
\begin{align*}
 \half\Psi\l(\|x\|\r)+\half\Psi\l(\|y\|\r)-\Psi\l(\l\|\hlf{x+y}\r\|\r) & \ge \Psi\l(\hlf{\|x\| +\|y\|}\r)-\Psi\l(\hlf{\|x\|+\|y\|}-\|y\|\delta_X\l(\frac{\eps}{4\|y\|}  \r)\r) \\ & \ge \|y\| \delta_X\l(\frac{\eps}{4\|y\|}\r)\Psi_+'\l(\hlf{\|y\|}\r) \\ & \ge \hlf{\|y\|} \delta_X\l( \frac{\frac{\eps}{8}}{\hlf{\|y\|}}\r) \Psi_+'\l(\hlf{\|y\|}\r) \\ & \ge \xi_{\frac{\eps}{8}}.
\end{align*}
The proof is complete.
\end{proof}

In the case of power-type uniform convexity, Proposition~\ref{prop3.14} implies the following chain of corollaries. The first of these corollaries (originally proved by Borwein and Vanderwerff in \cite[Theorem 2.3]{bv}) can easily be obtained as a special case of Proposition~\ref{prop3.14}.
\begin{cor} \label{cor:bv}
 Let $\l(X,\|\cdot \| \r)$ be a Banach space and $\Psi\wgt$ be convex nondecreasing. Assume $p\ge2$ and $\Psi$ and $\|\cdot\|$ have moduli of uniform convexity of power type~$p.$ If $\Psi_+'(t)\ge K t^{p-1}$ for some $K>0$ and every $t>0,$ then $\Psi\circ\|\cdot\|$ has a modulus of uniform convexity of power type~$p.$
\end{cor}
\begin{proof} One has a modulus of convexity of the form $\delta_X(\eps)=A\eps^p$, where we may 
assume that $A\le \half$ so that $\delta_X(\eps)\le \half$ for $\eps\in (0,1].$ 
Let $\delta_{\Psi}(\eps)=B\eps^p.$ Then we have
\begin{equation*}
\Psi_+'(t)\delta_X\l(\frac{\eps}{t}\r)t \ge AK\eps^p, \qquad\text{for every } t\ge \hlf{\eps},
\end{equation*}
and hence we may consider~\eqref{eq:witness} with $\xi_{\eps}\as AK\eps^p,$ for every $t\ge\hlf{\eps},$ and we may take $K_{\varepsilon}\as 0.$  Hence, in the proof of Theorem \ref{prop3.14}, `Case 1' cannot eventuate. Thus one obtains 
\begin{equation*}
  \delta_{\Psi\circ\|\cdot\|}(\eps)\as\min\l\{ \frac{B}{2^p},\frac{AK}{8^p}\r\}\eps^p,
 \end{equation*}
for every $\eps>0,$ which concludes the proof.
\end{proof}
One can use Corollary~\ref{cor:bv} to show the implication \eqref{i:xu:i}$\implies$\eqref{i:xu:ii} of Xu's theorem (Theorem \ref{thm:xu} above). This was observed by Borwein and Vanderwerff in \cite[Corollary 2.4]{bv}. Corollary~\ref{cor:bv} also provides a~modulus of uniform convexity of the function $x\mapsto\rec{p}\|x\|^p,$ that is, a quantitative version of the implication \eqref{i:xu:i}$\implies$\eqref{i:xu:ii} of Xu's theorem.
\begin{cor} \label{cor:pcase}
Suppose that $\l(X,\|\cdot \| \r)$ has a modulus of uniform convexity $\delta_X(\eps)= A\eps^p,$ for some $A\in\l(0,\half\r)$ and $p\ge2$ and every $\eps\in(0,2].$ Then the function $x\mapsto\rec{p}\|x\|^p$ has a modulus of uniform convexity
\begin{equation*}
 \delta_{\rec{p}\|\cdot\|^p}(\eps)\as\frac{A}{8^p}\eps^p,
\end{equation*}
for every $\eps>0.$
\end{cor}
\begin{proof}
By a result of Z{\u{a}}linescu from \cite[Proposition 3.2]{zalinescu83}, we can take
 \begin{equation*}
  \delta_{\rec{p}|\cdot|^p}(\eps)\as\frac{\eps^p}{p^2 2^{\frac{p^2-2p}{p-1}}},
 \end{equation*}
for every $\eps>0,$ as a modulus of uniform convexity. Then in the proof of Corollary~\ref{cor:bv} one can put $K=1$ and $B=p^{-2} 2^{-\frac{p^2-2p}{p-1}}.$ Since
\begin{equation*}
 \frac{B}{2^p}=\rec{p^2 2^{\frac{2p^2-3p}{p-1}}}\ge\rec{p^2 2^{\frac{2p^2-2p}{p-1}}}\ge \rec{8^p}\ge \frac{A}{8^p},
\end{equation*}
we obtain the desired modulus in Corollary~\ref{cor:bv}.
\end{proof}

With Corollary~\ref{cor:pcase} at hand, we are able to obtain a more concrete version of Theorem~\ref{thm:uc} in the case when $\Phi(t)=\rec{p}t^p$ for some $p\ge2.$ Note that if $\|\cdot\|$ is a uniformly convex norm and $p\in(1,2),$ then~$\|\cdot\|^p$ is uniformly convex on bounded sets, but not on the whole space. 
\begin{cor} \label{cor:hoelder}
 Let $\l(X,\|\cdot\|\r)$ be a Banach space with a modulus of uniform convexity $\delta_X(\eps)\as A\eps^p,$ for each $\eps\in (0,2],$ with some constants $A\in\l(0,\half\r)$ and $p\ge2.$ Let $\Phi(t)\as\rec{p}t^p.$ Then, given $z\in X$ and $r>0,$ there exists a~constant $L>0$ such that
 \begin{equation*}
  \l\|\prox_{\lam,f}^\Phi(x) -\prox_{\lam,f}^\Phi(y) \r\| \le \max\l\{2\|x-y\|,  L\|x-y\|^\rec{p} \r\},
\end{equation*}
for every $x,y\in B(z,r)$ and $\lam\in(0,1].$ If $r\ge\max\l\{1,\l\|z-\prox_{\lam_0,f}^\Phi(z) \r\|\r\},$ where $\lam_0\as1,$ we can set $L\as16r\l(\frac{3p+2^p}{2A}\r)^{\frac1p}.$
\end{cor}
\begin{proof}
If $\phi(t)=t^{p-1},$ then the modulus $\delta(\eps)$ from Theorem~\ref{thm:uc} is, on account of Corollary~\ref{cor:pcase}, equal to
\begin{equation} \label{eq:modl2}
 \delta(\eps)=\min\l\{\hlf{\eps},\frac{2A}{16^pR^{p-1}} \eps^p\r\},\qquad \eps>0,
\end{equation}
for some $R>0.$

Let $x,y\in B(z,r)$ with $\|x-y\|>0.$ By the continuity of the modulus~$\delta$ from~\eqref{eq:modl2}, we can find $\eps>0$ such that $\|x-y\|=\delta(\eps).$ Invoking that~$\delta$ is strictly increasing, we observe that implication~\eqref{eq:mainimpl} in Theorem~\ref{thm:uc} holds with non-strict inequalities as well and therefore we get, for each $\lam\in(0,1],$
\begin{equation*}
 \l\|\prox_{\lam,f}^\Phi(x) -\prox_{\lam,f}^\Phi(y) \r\| \le \max\l\{2\|x-y\|,  L\|x-y\|^\rec{p} \r\},
\end{equation*}
where $L\as16\l(\frac{R^{p-1}}{2A} \r)^\rec{p}$ and the constant~$R$ is independent of~$\lam.$
 
Furthermore, if $r\ge\max\l\{1,\l\|z-\prox_{\lam_0,f}^\Phi(z) \r\|\r\},$ the radius $R$ can be estimated from~\eqref{eq:optk1}, and one can therefore set $L\as16r\l(\frac{3p+2^p}{2A}\r)^{\frac1p}.$ 
\end{proof}

Next we turn to a renorming theorem from \cite[Theorem 2.4]{bghv}.
\begin{thm} \label{thm:renorming}
 Let $X$ be a Banach space. The following condition are equivalent.
 \begin{enumerate}
  \item There exists a continuous uniformly convex function defined on~$B_X.$ \label{i:renorm:i}
  \item There exists an equivalent uniformly convex norm on~$X.$ \label{i:renorm:ii}
  \item There exist an equivalent norm $|\cdot|$ on~$X$ and $p\ge2$ such that the function $f\as|\cdot|^p$ is uniformly convex. \label{i:renorm:iii}
 \end{enumerate}
\end{thm}
\begin{proof}
Let us outline the proof.

 \eqref{i:renorm:i}$\implies$\eqref{i:renorm:ii}: See \cite[Theorem 2.4]{bghv}, or the proof of Theorem~\ref{thm:qrenorming} below.
 
 \eqref{i:renorm:ii}$\implies$\eqref{i:renorm:iii}: By Pisier's renorming theorem~\cite{pisier}, mentioned above in Section~\ref{sec:pre}, there exists an equivalent norm $|\cdot|$ with modulus of uniform convexity of power type~$p,$ for some $p\ge2.$ Then Xu's result \cite[Theorem 1]{xu}, stated above as Theorem~\ref{thm:xu}, yields that $f\as|\cdot|^p$ is uniformly convex.
 
 \eqref{i:renorm:iii}$\implies$\eqref{i:renorm:i}: Trivial.
\end{proof}

To complete our quantitative analysis, we provide an explicit modulus of uniform convexity of the new norm in Theorem \ref{thm:renorming}\eqref{i:renorm:ii}. This is achieved by a straightforward modification of the original proof from~\cite{bghv}.
\begin{thm} \label{thm:qrenorming}
 Let $(X,\|\cdot\|)$ be a Banach space and $f\col B_X\to\exrls$ be a function which is uniformly convex on~$B_X$ with a modulus of uniform convexity $\delta_{f,B_X}$ and which is continuous at~$0$ with a modulus of continuity $\omega_{f,0}.$ Then there exist a constant $M>0$ and an equivalent norm $\nnorm{\cdot}$ on~$X$ which makes $\l(X,\nnorm{\cdot}\r)$ uniformly convex with a modulus of uniform convexity
 \begin{equation} \label{eq:moducnn}
  \delta_X(\eps)\as \frac{\omega_{f,0}(M)}{4M\alpha}\delta_{f,B_X}(\beta\eps), \qquad \eps\in (0,2],
 \end{equation}
and satisfying
\begin{equation} \label{eq:distortion}
 \rec{\alpha}\|x\| \le \nnorm{x} \le \rec{\beta}\|x\|,
\end{equation}
for every $x\in X.$ Here $\alpha\as\hlf{\omega_{f,0}(M)}$ and $\beta\as \omega_{f,0}\l(\delta_{f,B_X}(\alpha)\r).$
\end{thm}
\begin{proof}
Without loss of generality, we may assume that~$f$ is symmetric (upon replacing it with $x\mapsto\half\l(f(x)+f(-x)\r)$) and that $f(0)=0.$

The continuity of~$f$ at~$0$ gives $f(x)<\eps$ for every $x\in B\l(0,\omega_{f,0}(\eps)\r)$ and $\eps>0.$ Choose $M>0$ such that $\omega_{f,0}(M)<1.$ Then $f$ is on $B\l(0,\omega_{f,0}(M) \r)$ bounded by~$M.$ Define $\alpha\as\hlf{\omega_{f,0}(M)}$ and $\beta\as \omega_{f,0}\l(\delta_{f,B_X}(\alpha)\r).$ It is well known (see for instance \cite[Proposition 1.6]{phelps}) that $f$ is Lipschitz around~$0.$ More precisely, it is Lipschitz on $B\l(0,\hlf{\omega_{f,0}(M)} \r)$ with Lipschitz constant $\frac{4M}{\omega_{f,0}(M)}.$ Indeed, choose $x,y\in B(0,\alpha),$ $x\not= y,$ and 
denote for a moment $\gamma\as\|x-y\|$ and $z\as y+\frac{\alpha}{\gamma}(y-x).$ Then $z\in B\l(0,\omega_{f,0}(M)\r)$ and 
\begin{equation*}
 y=\frac{\gamma}{\gamma+\alpha}z+\frac{\alpha}{\gamma+\alpha}x.
\end{equation*}
By convexity,
\begin{equation*}
 f(y)-f(x)\le \frac{\gamma}{\gamma+\alpha}\l[f(z)-f(x)\r]\le\frac{\gamma}{\alpha}2M=\frac{4M}{\omega_{f,0}(M)}\|x-y\|.
\end{equation*}
Interchanging~$x$ and~$y$ yields the desired Lipschitz property.

Define
\begin{equation*}
 B\as\l\{x\in B_X\col f(x)\le \delta_{f,B_X}(\alpha)\r\}
\end{equation*}
and observe that
\begin{equation} \label{eq:nestedballs}
 B(0,\beta)\subset B \subset B(0,\alpha).
\end{equation}
Indeed, the first inclusion in~\eqref{eq:nestedballs} follows from the very definition of $\omega_{f,0}.$ And since for $u\in B_X$ with $\|u\|>\alpha$ we have
\begin{equation} \label{eq:incl}
 f(u)\ge 2\l[\half f(u)+\half f(0)- f\l(\hlf{u}\r) \r]\ge2\delta_{f,B_X}(\alpha)>\delta_{f,B_X}(\alpha),
\end{equation}
we obtain the second inclusion in~\eqref{eq:nestedballs}, too. 

Let us now define a new norm $\nnorm{x}\as\inf\l\{t>0\col x\in tB \r\},$ for $x\in X.$ By virtue of~\eqref{eq:nestedballs} we have
\begin{equation*}
 \rec{\alpha}\|x\|\le\nnorm{x}\le\rec{\beta}\|x\|,
\end{equation*}
for each $x\in X,$ which gives~\eqref{eq:distortion}.

Next choose $x,y\in X$ such that $\nnorm{x},\nnorm{y}\le1$ and assume $\nnorm{x-y}\ge\eps$ for some $\eps\in(0,2].$ Then $\|x-y\|\ge\beta\eps$ and hence
\begin{equation*}
 \delta_{f,B_X}(\beta\eps) \le \half f(x)+\half f(y)- f\l(\hlf{x+y}\r) \le \delta_{f,B_X}(\alpha)-f\l(\hlf{x+y}\r) .
\end{equation*}
Since~\eqref{eq:incl} holds also for $u\in X$ with $\|u\|=\alpha,$ and since $f$ is Lipschitz on $B(0,\alpha),$ we conclude that $f(v)=\delta_{f,B_X}(\alpha)$ whenever $\nnorm{v}=1.$ Therefore,
\begin{equation*}
 \delta_{f,B_X}(\alpha)-f\l(\hlf{x+y}\r)=f\l(\frac{x+y}{\nnorm{x+y}}\r)-f\l(\hlf{x+y}\r).
\end{equation*}
Recalling again the Lipschitz property of~$f$ yields
\begin{equation*}
 \delta_{f,B_X}(\beta\eps)\le \nnorm{\frac{x+y}{\nnorm{x+y}}-\hlf{x+y}} \frac{4M}{\omega_{f,0}(M)}\alpha= \l( 1-\nnorm{\hlf{x+y}}\r)\frac{4M}{\omega_{f,0}(M)}\alpha.
\end{equation*}
Hence
\begin{equation*}
   \delta_X(\eps)\as \frac{\omega_{f,0}(M)}{4M\alpha}\delta_{f,B_X}(\beta\eps),
\end{equation*}
is a modulus of uniform convexity of $\l(X,\nnorm{\cdot}\r),$ which gives the remaining property~\eqref{eq:moducnn} and the proof is complete.
\end{proof}

In their recent paper~\cite{gjy}, Gonzalo, Jaramillo and Y\'{a}\~{n}ez showed that a polynomial norm has a power-type modulus of uniform convexity. We now state their result, and since the proof in \cite[Proposition 4]{gjy} is slightly inaccurate (namely, their set $\mathcal{C}_N$ does contain the zero polynomial), we present a corrected proof. Let us first introduce polynomial norms. Let $\l(X,\|\cdot\|\r)$ be a Banach space and $N\in\nat$ an even integer. If there exists a continuous symmetric $N$-linear form $A\col X^N\to\rls$ such that $\|x\|^N=A(x,\dots,x)$ for each $x\in X,$ we say that $\|\cdot\|$ is a \emph{polynomial norm.} The diagonal of~$A,$ that is $x\mapsto P(x)\as A(x,\dots,x)$ is called an $N$-homogeneous polynomial.
\begin{prop}[Gonzalo, Jaramillo, Y\'{a}\~{n}ez]
 The norm $\|\cdot\|$ has a modulus of uniform convexity of power type~$N.$
\end{prop}
\begin{proof}
By Xu's theorem (Theorem~\ref{thm:xu} above), we need to show that
\begin{equation*}
 \inf\l\{\half P(x)+\half P(y)-P\l(\hlf{x+y} \r)\col x,y\in X,\|x-y\|=t \r\}>0,
\end{equation*}
for each $t>0.$ Using a substitution $z\as\hlf{x+y}$ and $h\as\frac{x-y}{2t},$ the above inequality is equivalent to
\begin{equation} \label{eq:infpoly}
 \inf\l\{P(z+th)+P(z-th)-2P(z)\col z,h\in X,\|h\|=\half \r\}>0.
\end{equation}
Given $z,h\in X$ such that $\|h\|=\half,$ denote
\begin{equation*}
 p_{z,h}(t)\as P(z+th)+P(z-th)-2P(z),\qquad t\in\rls,
\end{equation*}
which is a polynomial belonging to the set
\begin{equation*}
 W\as\l\{p(t)=a_Nt^N+a_{N-2}t^{N-2}+\cdots+a_2t^2\col p\ge0, p \text{ convex },a_N=\rec{2^{N-1}} \r\},
\end{equation*}
since the leading term of $p_{z,h}$ is $A(h,\ldots,h)t^N+A(h,\ldots,h)t^N$ which is equal to $2\| h\|^Nt^N.$ 
Given $t_0>0,$ we have
\begin{equation*}
 p_{z,h}\l(t_0\r)\ge \inf \l\{p\l(t_0\r)\col p\in W \r\}>0,
\end{equation*}
for every $z,h\in X$ such that $\|h\|=\half,$ which implies that~\eqref{eq:infpoly} holds true.
\end{proof}


\bibliographystyle{siam}
\bibliography{prox}

\end{document}